\documentclass[a4paper,11pt,reqno]{customart}

\usepackage[utf8x]{inputenc}
\usepackage[margin=1.5in]{geometry}
\usepackage{verbatim}
\usepackage{color}
\usepackage[table]{xcolor}
\usepackage{listings}
\usepackage{amssymb,amsfonts,amsthm,amsmath}
\usepackage{url}
\usepackage{enumerate}
\usepackage{todonotes}
\usepackage[all]{xy}
\usepackage{multirow}
\usepackage{tikz}
\usepackage[underline=false]{pgf-umlsd}
\usepackage{stmaryrd}
\usepackage{footnote}
\makesavenoteenv{tabular}
\usepackage{hyperref}

\makeatletter
\newcommand\footnoteref[1]{\protected@xdef\@thefnmark{\ref{#1}}\@footnotemark}
\makeatother

\hypersetup{
	pdfborder={0 0 0}
}

\definecolor{grey}{rgb}{0.95,0.95,0.95}
\definecolor{green}{rgb}{0.2,0.6,0.4}

\renewcommand{\O}{\textbf{0}}

\newcommand{\Psf}{\mathsf{P}}
\newcommand{\Qsf}{\mathsf{Q}}

\newcommand{\Bcal}{\mathcal{B}}
\newcommand{\Ccal}{\mathcal{C}}
\newcommand{\Dcal}{\mathcal{D}}

\newcommand{\Ucal}{\mathcal{U}}

\newcommand{\Ecal}{\mathcal{E}}

\newcommand{\Ocal}{\mathcal{O}}

\newcommand{\Qcal}{\mathcal{Q}}

\newcommand{\tuple}[1]{\left\langle #1 \right\rangle}

\newcommand{\s}[1]{\ensuremath{\sf{#1}}}

\DeclareMathOperator{\rca}{\s{RCA}_0}

\DeclareMathOperator{\wkl}{\s{WKL}}
\DeclareMathOperator{\wwkl}{\s{WWKL}}

\DeclareMathOperator{\rt}{\s{RT}}
\DeclareMathOperator{\srt}{\s{SRT}}

\DeclareMathOperator{\kl}{\s{KL}}

\usetikzlibrary{shapes,arrows}
\usetikzlibrary{decorations.markings}
\definecolor{lightblue}{HTML}{e6e6e6}
\definecolor{lightred}{HTML}{eca6a6}
\definecolor{lightgreen}{RGB}{164,244,140}

\newtheoremstyle{custom}
  {10pt}
  {10pt}
  {\normalfont}
  {}
  {\bfseries}
  {}
  { }
  {}

\theoremstyle{custom}

\usepackage{xcolor}	
\usepackage{soul}

\newtheorem{theorem}{Theorem}[section]
\newtheorem{lemma}[theorem]{Lemma}
\newtheorem{definition}[theorem]{Definition}

\newtheorem{question}[theorem]{Question}

\newtheorem{corollary}[theorem]{Corollary}

\begin{document}

\title{$\Pi^0_1$ encodability and omniscient reductions}
\author{
  Benoit Monin \and Ludovic Patey
}

\begin{abstract}
A set of integers~$A$ is computably encodable if every infinite set of integers has an infinite subset computing~$A$.
By a result of Solovay, the computably encodable sets are exactly the hyperarithmetic ones.
In this paper, we extend this notion of computable encodability 
to subsets of the Baire space and we characterize the $\Pi^0_1$ encodable compact sets 
as those who admit a non-empty $\Sigma^1_1$ subset.
Thanks to this equivalence, we prove that weak weak K\"onig's lemma is not strongly computably reducible
to Ramsey's theorem. This answers a question of Hirschfeldt and Jockusch.
\end{abstract}

\maketitle

\section{Introduction}

A set~$A \subseteq \omega$ is \emph{computably encodable} if every infinite set~$X \subseteq \omega$ has an infinite subset
computing~$A$. Jockusch and Soare~\cite{Jockusch1973Encodability} introduced various notions of encodability
and Solovay~\cite{Solovay1978Hyperarithmetically} characterized the computably encodable sets as the hyperarithmetical ones.
We extend the notion of computable encodability to collections of sets as follows.
A set~$\Ccal \subseteq \omega^\omega$ is \emph{$\Pi^0_1$ encodable} if every infinite set
$X \subseteq \omega$ has an infinite subset $Y$ such that 
$\Ccal$ admits a non-empty $Y$-computably bounded $\Pi^{0,Y}_1$ subset $\Dcal \subseteq \omega^\omega$.
By this, we mean that $\Dcal = [T]$ for some $Y$-computable tree~$T$ whose nodes are bounded by a $Y$-computable function.
Our main result asserts that the compact sets that are $\Pi^0_1$ encodable are exactly those 
admitting a non-empty $\Sigma^1_1$ subset. This extends Solovay's theorem as 
the members of the $\Sigma^1_1$ singletons and these of the computably bounded $\Pi^0_1$ singletons
are exactly the hyperarithmetic ones~\cite{Spector1955Recursive} and the computable ones, respectively.
Our motivations follow two axis. 

First, the development of \emph{mass problems}
such as Muchnik and Medvedev degrees~\cite{Hinman2012survey} revealed finer computational behaviors
than those captured by the Turing degrees. For example, the cone avoidance basis theorem~\cite{Jockusch197201}
asserts that the PA degrees are of no help to compute a single incomputable set of integers.
However, it would be simplistic to deduce that PA degrees carry no computational power.
For example, they enable one to compute separating sets given two computably inseparable c.e.\ sets.
This work can therefore be seen as part of a program of extending core computability-theoretic theorems about Turing degrees
to their generalized statements about mass problems.

Our second motivation comes from the reverse mathematics and the computable analysis of Ramsey's theorem.
Computable encodability is a very important feature of Ramsey's theorem, as for every $k$-coloring of~$[\omega]^n$,
and every infinite set~$X$, there is an infinite homogeneous subset contained in~$X$. Computable encodability
provides a formal setting to many intuitions about the computational weakness of Ramsey's theorem.
In particular, we use this notion to answer a question asked by Hirschfeldt and Jockusch~\cite{Hirschfeldtnotions} about the link
between variants of K\"onig's lemma and Ramsey's theorem over strong computable reducibility.

\subsection{Reductions between mathematical problems}

A \emph{mathematical problem} $\Psf$ is specified by a collection of \emph{instances},
coming together with a collection of \emph{solutions}.
Many ordinary theorems can be seen as mathematical problems. For example, K\"onig's lemma ($\kl$)
asserts that every infinite, finitely branching tree admits an infinite path.
In this setting, an instance of $\kl$ is an infinite, finitely branching tree~$T$,
and a solution to~$T$ is any infinite path~$P \in [T]$.

There are many ways to compare the strength of mathematical problems.
Among them, reverse mathematics study their logical consequences~\cite{Simpson2009Subsystems}.
More recently, various notions of effective reductions have been proposed to compare mathematical problems,
namely, Weihrauch reductions~\cite{Brattka2015Uniform,Dorais2016uniform}, computable reductions~\cite{Hirschfeldtnotions}, 
computable entailment~\cite{Shore2010Reverse}, among others.
A problem~$\Psf$ is \emph{computably reducible} to another problem~$\Qsf$ (written $\Psf \leq_c \Qsf$) if every $\Psf$-instance~$I$
computes a $\Qsf$-instance~$J$ such that every solution to~$J$ computes relative to~$I$ a solution to~$I$.
$\Psf$ is \emph{Weihrauch reducible} to $\Qsf$ (written $\Psf \leq_W \Qsf$) if moreover 
this computable reduction is witnessed by two fixed Turing functionals.
There exist strong variants of computable and Weihrauch reductions written $\Psf \leq_{sc} \Qsf$
and $\Psf \leq_{sW} \Qsf$, respectively, where no access to the $\Psf$-instance~$I$ is allowed in the backward reduction.
In this article, we shall focus on strong computable reduction.

Due to the range of potential definitions of effective reductions, there is a need to give a justification
about the choices of the definition. 
An effective reduction from~$\Psf$ to~$\Qsf$ should reflect some computational aspect of the relationship between $\Psf$ and $\Qsf$.
The more precise the reduction is, the more insights it gives about the links between the two problems.
As it happens, many proofs that~$\Psf$ is not strongly computably reducible to~$\Qsf$ actually produce
a single $\Psf$-instance~$I$ such that for every $\Qsf$-instance~$J$, computable in~$I$ or not,
there is a solution to~$J$ computing no solution to~$I$. Such a relation suggests a deep structural 
difference between the problems~$\Psf$ and~$\Qsf$, in that even with a perfect knowledge of~$I$,
there is no way to encode enough information in the~$\Qsf$-instance to solve~$I$.
We shall therefore define $\Psf$ to be \emph{strongly omnisciently computably reducible} to~$\Qsf$ (written $\Psf \leq_{soc} \Qsf$)
if for every $\Psf$-instance~$I$, there is a $\Qsf$-instance~$J$ such that every solution to~$J$ computes a solution to~$I$.

\subsection{K\"onig's lemma and Ramsey's theorem}\label{subsect:kl-and-rt}

K\"onig's lemma and Ramsey's theorem are core theorems from mathematics,
both enjoying a special status in reverse mathematics.

\begin{definition}[Various K\"onig lemmas]
$\kl$ is the statement ``Every infinite finitely-branching tree has an infinite path''.
$\wkl$ is the restriction of~$\kl$ to binary trees.
$\wwkl$ is the restriction of~$\wkl$ to binary trees of positive measure (A binary tree~$T \subseteq 2^{<\omega}$ has \emph{positive measure}
if $\lim_s \frac{|\{\sigma \in T : |\sigma| = s\}|}{2^s} > 0$). 
\end{definition}

Weak K\"onig's lemma captures compactness arguments 
and naturally arises from the study of ordinary theorems~\cite{Simpson2009Subsystems}. 
It is part of the so called \emph{Big Five}~\cite{Montalban2011Open}.
On the other hand, weak weak K\"onig's lemma can be thought of as
asserting the existence of randomness in the sense of Martin-L\"of~\cite{Downey2010Algorithmic}.
Although weak K\"onig's lemma is strictly weaker than K\"onig's lemma in reverse mathematics
and over computable reducibility, the statements are trivially equivalent over strong omniscient computable reducibility.
Indeed, given any problem~$\Psf$ admitting an instance with at least one solution~$S$,
one can define a binary tree whose unique path is a binary coding of~$S$. In particular, $\kl \leq_{soc} \wkl$.
Weak weak K\"onig's lemma, as for him, remains strictly weaker than K\"onig's lemma over
strong omniscient computable reducibility, since the measure of the set of oracles computing a non-computable
set is null~\cite{Sacks1963Degrees}. Therefore one can choose any tree with a unique incomputable path
as an instance of K\"onig's lemma to show that $\kl \nleq_{soc} \wwkl$

\begin{definition}[Ramsey's theorem]
A subset~$H$ of~$\omega$ is~\emph{homogeneous} for a coloring~$f : [\omega]^n \to k$ (or \emph{$f$-homogeneous}) 
if each $n$-tuples over~$H$ are given the same color by~$f$. 
$\rt^n_k$ is the statement ``Every coloring $f : [\omega]^n \to k$ has an infinite $f$-homogeneous set'',
$\rt^n_{<\infty}$ is $(\forall k)\rt^n_k$ and~$\rt$ is $(\forall n)\rt^n_{<\infty}$.
\end{definition}

Ramsey's theorem received a lot of attention in reverse mathematics
since it is one of the first examples of statements escaping the Big Five phenomenon.
There is profusion of litterature around the reverse mathematics and computable analysis of 
Ramsey's theorem~\cite{Jockusch1972Ramseys,Seetapun1995strength,Cholak2001strength,Hirschfeldt2008strength}. 
In particular, $\rt^n_k$ is equivalent to~$\kl$ in reverse mathematics
for any standard~$n \geq 3$ and~$k \geq 2$~\cite{Simpson2009Subsystems} and $\rt^2_k$ is strictly 
in between $\rca$ and~$\rt^3_k$~\cite{Seetapun1995strength}.
More recently, there has been studies of Ramsey's theorem under various notions of reducibility.
Let~$\srt^2_k$ denote the restriction of $\rt^2_k$ to \emph{stable} colorings,
that is, functions~$f : [\omega]^2 \to k$ such that~$\lim_s f(x,s)$ exists for every~$x$.
In what follows, $k \geq 2$.
Brattka and Rakotoniaina~\cite{Brattka2015Uniform} and Hirschfeldt and Jockusch~\cite{Hirschfeldtnotions}
studied the Weihrauch degrees of Ramsey's theorem and independently proved that~$\rt^1_{k+1} \not \leq_W \srt^2_k$
and~$\rt^n_{<\infty} \leq_{sW} \rt^{n+1}_2$. Note that the reduction $\rt^1_{k} \leq_{sW} \srt^2_k$ trivially holds.

From the point of view of omniscient reductions, the above discussion about weak K\"onig's lemma shows that~$\rt \leq_{soc} \wkl$.
Dzhafarov and Jockusch~\cite{Dzhafarov2009Ramseys} proved that~$\srt^2_2 \not \leq_{soc} \rt^1_{<\infty}$.
Hirschfeldt and Jockusch~\cite{Hirschfeldtnotions} and Patey~\cite{Patey2016weakness} independently proved
that~$\rt^1_{k+1} \not \leq_{soc} \rt^1_k$, later strengthened by
Dzhafarov, Patey, Solomon and Westrick~\cite{Dzhafarov2016Ramseys}, who proved that~$\rt^1_{k+1} \not \leq_{soc} \srt^2_k$
and that $\rt^2_2 \not \leq_{soc} \srt^2_{<\infty}$.
Some differences between strong computable reducibility and strong omniscient computable reducibility
are witnessed by Ramsey's theorem. For example, the second author~\cite{Patey2016weakness} 
proved that~$\srt^2_{k+1} \not \leq_{sc} \rt^2_k$,
while the following argument shows that~$\srt^2_{<\infty} \leq_{soc} \rt^2_2$: Given a stable coloring~$f : [\omega]^2 \to k$,
let $g(x, y) = 1$ iff $f(x, y) = \lim_s f(y, s)$. Every infinite $g$-homogeneous set is for color 1 and
is an $f$-homogeneous set.

Hirschfeldt and Jockusch compared Ramsey's theorem and K\"onig's lemma over strong omniscient computable reducibility
and proved that $\rt^1_2 \not \leq_{soc} \wwkl$ and that~$\wkl \not \leq_{soc} \rt$.
They asked whether weak weak K\"onig's lemma is a consequence of Ramsey's theorem over strong computable reducibility.
We answer negatively by proving the stronger separation~$\wwkl \not \leq_{soc} \rt$.

\subsection{Notation}

Given a set~$A$ and some integer~$n \in \omega$, we let~$[A]^n$ denote the collection of all unordered subsets of~$A$ of size~$n$.
Accordingly, we let $A^{<\omega}$ and $[A]^{\omega}$ denote the collection of all finite and infinite subsets of $A$, respectively.
Given $a \in [\omega]^{<\omega}$ and $X \in [\omega]^{\omega}$ such that~$\max a < \min X$,
we let $\tuple{a, X}$ denote the set of all $B \in [\omega]^\omega$
such that $a \subseteq B \subseteq a \cup X$.
The pairs $\tuple{a, X}$ are called \emph{Mathias conditions} and form,
together with $\emptyset$, the basic open sets of the \emph{Ellentuck topology}.

Given a function~$f \in \omega^\omega$ and an integer~$t \in \omega$,
we write~$f^t$ for the set of all strings~$\sigma \in \omega^{<\omega}$
of length~$t$ such that~$(\forall x < t)\sigma(x) \leq f(x)$.
Accordingly, we write~$f^{<\omega}$ for $\bigcup_{t \in \omega} f^t$.

\section{Computable encodability}

A function~$f \in \omega^\omega$ is a \emph{$\Pi^0_1$ modulus} of a set~$\Ccal \subseteq \omega^\omega$
if $\Ccal$ has a non-empty $g$-computably bounded $\Pi^{0,g}_1$ subset for every function~$g \geq f$.
A function~$f \in \omega^\omega$ is a modulus of a set~$A \in \omega^\omega$
if $g \geq_T A$ for every~$g \geq f$.
Note that the notion of $\Pi^0_1$ modulus of the singleton~$\{A\}$ coincides with the existing notion of modulus of the set~$A$
since the members of computably bounded $\Pi^0_1$ singletons are computable.
The purpose of this section is to prove the following main theorem.

\begin{theorem}\label{thm:comp-enc-compact-sigma11}
Fix a compact set~$\Ccal \subseteq \omega^\omega$. The following are equivalent:
\begin{itemize}
	\item[(i)] $\Ccal$ is $\Pi^0_1$ encodable
	\item[(ii)] $\Ccal$ admits a $\Pi^0_1$ modulus
	\item[(iii)] $\Ccal$ has a non-empty $\Sigma^1_1$ subset
\end{itemize}
\end{theorem}
\begin{proof}
$(ii) \Rightarrow (i)$: Let~$f$ be a $\Pi^0_1$ modulus of~$\Ccal$.
For every set~$X \in [\omega]^\omega$, there is a set~$Y \in [X]^\omega$
such that~$p_Y \geq f$, where $p_Y(x)$ is the $x$th element of~$Y$ in increasing order. 
In particular, $\Ccal$ has a non-empty $\Pi^{0,Y}_1$ subset.
$(iii) \Rightarrow (ii)$: Let~$R(X, Y, z)$ be a computable predicate
such that $\Dcal = \{X \in \omega^\omega : (\exists Y \in \omega^\omega)(\forall z)R(X, Y, z) \}$
is a non-empty subset of~$\Ccal$. Since $\Dcal \neq \emptyset$, 
there are some~$X, Y \in \omega^\omega$ such that $R(X, Y, z)$ holds for every~$z \in \omega$.
We claim that the function~$f$ defined by $f(x) = max(X(x), Y(x))$ is a $\Pi^0_1$ modulus of~$\Ccal$.
To see this, pick any function~$g \geq f$. The set
$
\{ X \leq g : (\forall z \in \omega)(\exists \rho \in g^z)(\forall y < z)R(X, \rho, y) \}
$
is a non-empty $\Pi^{0,g}_1$ subset of~$\Ccal$ bounded by~$g$.
The remainder of this section will be dedicated to the proof of $(i) \Rightarrow (iii)$.
\end{proof}

\begin{corollary}[Solovay~\cite{Spector1955Recursive}, Groszek and Slaman~\cite{Groszek2007Moduli}]
Fix a set~$A \in \omega^\omega$. The following are equivalent:
\begin{itemize}
	\item[(i)] $A$ is computably encodable
	\item[(ii)] $A$ admits a modulus
	\item[(iii)] $A$ is hyperarithmetic
\end{itemize}
\end{corollary}
\begin{proof}
By Theorem~\ref{thm:comp-enc-compact-sigma11}, it suffices to prove that~$A$ is computably encodable, admits
a modulus, and is hyperarithmetic if and only if $\{A\}$ is $\Pi^0_1$ encodable, admits a $\Pi^0_1$ modulus
and has a non-empty $\Sigma^1_1$ subset, respectively.

By Spector~\cite{Spector1955Recursive}, a set~$A \in \omega^\omega$ is hyperarithmetic iff it is the unique member
of a $\Sigma^1_1$ singleton set~$\Ccal \subseteq \omega^\omega$. Therefore, $A$ is hyperarithmetic iff $\{A\}$
has a non-empty $\Sigma^1_1$ subset.
Every modulus of $A \in \omega^\omega$ is a $\Pi^0_1$ modulus of $\{A\}$.
Conversely, if $\{A\}$ admits a $\Pi^0_1$ modulus~$f$, then for every~$g \geq f$, 
$\{A\}$ is a $g$-computably bounded $\Pi^{0,g}_1$ singleton, 
so $A$ is $g$-computable. Therefore~$f$ is a modulus of~$A$.
If~$A$ is computably encodable, then $\{A\}$ is $\Pi^0_1$ encodable since every $X$-computable set is an $X$-computably bounded
$\Pi^{0,X}_1$ singleton.
Conversely, suppose that~$\{A\}$ is $\Pi^0_1$ encodable. Then, for every set~$X \in [\omega]^\omega$,
there is a set~$Y \in [X]^\omega$ such that~$\{A\}$ is a $Y$-computably bounded $\Pi^0_1$ class.
In particular, $Y$ computes $A$.
\end{proof}

A \emph{basis} for the $\Sigma^1_1$ sets is a collection of sets~$\Bcal \subseteq \omega^\omega$
such that~$\Bcal \cap \Dcal \neq \emptyset$ for every non-empty $\Sigma^1_1$ set $\Dcal \subseteq \omega^\omega$.
Gandy, Kreisel and Tait~\cite{Gandy1960Set} proved that whenever a set~$A \in \omega^\omega$ is non-hyperarithmetic,
every non-empty $\Sigma^1_1$ set~$\Dcal \subseteq \omega^\omega$ has a member~$X$ such that~$A$ is not hyperarithmetic in~$X$.
We need to extend their basis theorem by replacing non-hyperarithmetic sets by compact sets with no non-empty $\Sigma^1_1$ subsets
in order to prove the remaining direction of Theorem~\ref{thm:comp-enc-compact-sigma11}.
Note that when we apply Theorem~\ref{thm:spine-basis-theorem} with~$\Ccal = \{A\}$
for some non-hyperarithmetic set~$A$, we get back the non-hyperarithmetic basis theorem
of Gandy, Kreisel and Tait.

\begin{theorem}[$\Sigma^1_1$-immunity basis theorem]\label{thm:spine-basis-theorem}
For every compact set $\Ccal \subseteq \omega^\omega$ with no non-empty $\Sigma^1_1$ subset,
and every non-empty $\Sigma^1_1$ set $\Dcal \subseteq \omega^\omega$,
there is some~$X \in \Dcal$ such that $\Ccal$ has no non-empty $\Sigma^{1,X}_1$ subset.
\end{theorem}

Theorem~\ref{thm:spine-basis-theorem} is an easy consequence of the following lemma.

\begin{lemma}\label{lem:spine-basis-theorem-iter}
Fix a compact set $\Ccal \subseteq \omega^\omega$ with no non-empty $\Sigma^1_1$ subset
and a $\Sigma^1_1$ predicate $P(X, Y)$.
Every non-empty $\Sigma^1_1$ set $\Dcal \subseteq \omega^\omega$ has a non-empty $\Sigma^1_1$ subset~$\Ecal$
such that $\{ Y \in \omega^\omega : P(X, Y) \} \not \subseteq \Ccal$ for every $X \in \Ecal$.
\end{lemma}
\begin{proof}
We reason by case analysis. In the first case, $\{ Y \in \omega^\omega : P(X, Y) \} \subsetneq \Ccal$
for some~$X \in \Dcal$. Let~$Y \not \in \Ccal$ be such that $P(X,Y)$ holds.
By closure of~$\Ccal$, there is some finite initial segment~$\sigma \prec Y$ such that~$[\sigma] \cap \Ccal = \emptyset$.
The $\Sigma^1_1$ set $\Ecal = \{ X \in \Dcal : (\exists Y \succ \sigma)P(X, Y) \}$ is non-empty
and satisfies the desired properties.
In the second case, for every~$X \in \Dcal$, $\{Y \in \omega^\omega : P(X, Y) \} \subseteq \Ccal$.
Then $\{ Y \in \omega^\omega : (\exists X \in \Dcal)P(X,Y) \}$ is a $\Sigma^1_1$ subset of~$\Ccal$,
and therefore must be empty. We can simply choose~$\Ecal = \Dcal$.
\end{proof}

\begin{proof}[Proof of Theorem~\ref{thm:spine-basis-theorem}]
Let us consider for any  $\Sigma^1_1$ predicate $P(X, Y)$, the union $\Ucal_P$ of all the $\Sigma^1_1$ sets $\Ecal$ such that $\{ Y \in \omega^\omega :P(X, Y) \} \not \subseteq \Ccal$ for every $X \in \Ecal$. By Lemma~\ref{lem:spine-basis-theorem-iter}, each $\Ucal_P$ is dense for the Gandy-Harrington topology (where open sets are those generated by the $\Sigma^1_1$ sets). It is well known that $\omega^\omega$ with the Gandy-Harrington topology is a Baire space. It follows that $\bigcap_P \Ucal_P$ is dense. In particular it has a non-empty intersection with any $\Sigma^1_1$ set. Also it is clear by the definition of $\Ucal_P$ that 
$\Ccal$ contains no $\Sigma^{1,X}_1$ subset for any $X \in \bigcap_P \Ucal_P$.
\end{proof}

We will now prove the core lemma from which we will deduce the last direction of Theorem~\ref{thm:comp-enc-compact-sigma11}.
In what follows, we assume that $\Gamma$ ranges over trees,
that is, if $\Gamma^v(\tau) \downarrow$, then $\Gamma^v(\sigma) \downarrow$ for every~$\sigma \preceq \tau$.

\begin{lemma}\label{lem:encod-mass-set}
Fix a set $X \in [\omega]^\omega$ and a compact set~$\Ccal \subseteq \omega^\omega$
with no non-empty $\Sigma^{1,X}_1$ subset.
For every continuous function~$\Gamma$ and every $t \in \omega$, there
is a set~$Y \in [X]^\omega$ such that for every~$G \in [Y]^\omega$,
either $\Ccal \cap [\Gamma^G] = \emptyset$ or $\Gamma^G(\sigma)\downarrow = 1$
for some string $\sigma \in \omega^{<\omega}$ of length at least~$t$ such that $\Ccal \cap [\sigma] = \emptyset$.
Moreover, we can choose~$Y$ so that $\Ccal$ has no $\Sigma^{1,Y}_1$ subset.
\end{lemma}
\begin{proof}
Given $v \in [\omega]^{<\omega}$ and~$n \in \omega$, let~$S^v_n$ be the computable set 
$\{ \tau \in \omega^n : \Gamma^v(\tau) \downarrow \}$.
For every~$\sigma \in \omega^{<\omega}$, let~$\Qcal_\sigma$ be the $\Sigma^{1,X}_1$ collection of all~$Y \in [X]^\omega$
such that for every~$v \in [Y]^{<\omega}$, $S^v_{|\sigma|} = \emptyset$ or~$\sigma \in S^v_{|\sigma|}$.

Suppose first that for every~$\ell \in \omega$, there is some~$\sigma \in \omega^{<\omega}$ of length~$\ell$
such that $\Qcal_\sigma \neq \emptyset$. 
If $\Qcal_\sigma \neq \emptyset$ for some~$\sigma \in \omega^{<\omega}$ of length at least~$t$
such that~$\Ccal \cap [\sigma] = \emptyset$, then by Theorem~\ref{thm:spine-basis-theorem},
there is some~$Y \in \Qcal_\sigma$ such that $\Ccal$ has no non-empty $\Sigma^{1,Y}_1$ subset.
Such a $Y$ and~$\sigma$ satisfy the desired properties. If $\Ccal \cap [\sigma] \neq \emptyset$
for every~$\sigma \in \omega^{<\omega}$ of length at least~$t$ such that $\Qcal_\sigma \neq \emptyset$,
then by compactness of~$\Ccal$, the set
$\{ h \in \omega^\omega : (\forall \sigma \prec h)Q_\sigma \neq \emptyset \}$
is a non-empty  $\Sigma^{1,X}_1$ subset of~$\Ccal$, contradicting our hypothesis.

Suppose now that there is some~$\ell \in \omega$ such that~$\Qcal_\sigma = \emptyset$ for every~$\sigma \in \omega^{<\omega}$ of length $l$.
Let~$\sigma_0, \dots, \sigma_{n-1}$ be the finite sequence of all $\sigma \in \omega^\ell$ such that $\Ccal \cap [\sigma] \neq \emptyset$.
This sequence is finite by compactness of~$\Ccal$. Let~$\Ecal$ be the $\Sigma^{1,X}_1$ collection of all~$Y \in [X]^\omega$
such that for every~$v \in [Y]^{<\omega}$ and every~$i < n$, $\sigma_i \not \in S^v_\ell$. We claim that $\Ecal$ is non-empty.
To see this, define a finite decreasing sequence 
$X = X_0 \supseteq X_1 \supseteq \dots \supseteq X_n$ of infinite sets such that for every~$i < n$
and every~$v \in [X_{i+1}]^{<\omega}$, $\sigma_i \not \in S^v_\ell$ as follows.
Given~$i < n$, and since $\Qcal_{\sigma_i} = \emptyset$,
apply the Galvin-Prikry theorem~\cite{Galvin1973Borel} relative to~$X_i$
to obtain a set~$X_{n+1} \in [X_i]^\omega$ such that $\sigma_i \not \in S^v_\ell$ for every~$v \in [X_{i+1}]^{<\omega}$.
By Theorem~\ref{thm:spine-basis-theorem}, there is some~$Y \in \Ecal$ such that $\Ccal$ has no non-empty $\Sigma^{1,Y}_1$ subset.
Such a $Y$ satisfies the desired conditions. This completes the proof.
\end{proof}

\begin{lemma}\label{lem:encod-mass-mathias}
Fix a Mathias condition~$\tuple{a, X}$ and a compact set~$\Ccal \subseteq \omega^\omega$
with no non-empty $\Sigma^{1,X}_1$ subset.
For every continuous function~$\Gamma$ and every $t \in \omega$, there
is a condition~$\tuple{a, Y} \subseteq \tuple{a, X}$ such that for every~$G \in \tuple{a, Y}$
and every~$H \in [G]^\omega$,
either $\Ccal \cap [\Gamma^H] = \emptyset$ or $\Gamma^H(\sigma)\downarrow = 1$
for some string $\sigma \in \omega^{<\omega}$ of length at least~$t$ such that $\Ccal \cap [\sigma] = \emptyset$.
Moreover, we can choose~$Y$ so that $\Ccal$ has no $\Sigma^{1,Y}_1$ subset.
\end{lemma}
\begin{proof}
Let~$a_0, \dots, a_{n-1}$ be the finite listing of all subsets of~$a$,
and for every~$i < n$, let~$\Gamma_i$ be the continuous function defined by
$\Gamma_i^Z = \Gamma^{a_i \cup Z}$. By iterating Lemma~\ref{lem:encod-mass-set} on each~$\Gamma_i$,
we obtain a set~$Y \in [X]^\omega$ such that $\Ccal$ has no non-empty $\Sigma^{1,Y}_1$ subset,
and for every $Z \in [Y]^\omega$, and every~$i < n$, either $\Ccal \cap [\Gamma_i^Z] = \emptyset$ or $\Gamma_i^Z(\sigma)\downarrow = 1$
for some string $\sigma \in \omega^{<\omega}$ of length at least~$t$ such that $\Ccal \cap [\sigma] = \emptyset$.

We claim that $\tuple{a,Y}$ satisfies the desired properties. Fix any~$G \in \tuple{a,Y}$ and~$H \in [G]^\omega$.
In particular, $H = a_i \cup Z$ for some~$i < n$ and~$Z \in [Y]^\omega$.
Therefore, either~$\Ccal \cap [\Gamma^H] = \Ccal \cap [\Gamma_i^Z] = \emptyset$,
or $\Gamma^H(\sigma) \downarrow = \Gamma_i^Z(\sigma) = 1$ for some string 
$\sigma \in \omega^{<\omega}$ of length at least~$t$ such that $\Ccal \cap [\sigma] = \emptyset$.
\end{proof}

\begin{proof}[Proof of Theorem~\ref{thm:comp-enc-compact-sigma11}, $(i) \Rightarrow (iii)$]
We now prove that if a compact set~$\Ccal \subseteq \omega^\omega$ has no non-empty $\Sigma^1_1$ subset,
then there is a set~$Y \in [\omega]^\omega$ such that for every $G \in [Y]^\omega$, every $G$-computably bounded $\Pi^{0, G}_1$ set is not included in $\Ccal$.

By iterating Lemma~\ref{lem:encod-mass-mathias}, build an infinite sequence
of Mathias conditions $\tuple{\emptyset, \omega} = \tuple{a_0, X_0} \supseteq \tuple{a_1, X_1} \supseteq \dots$
such that for every~$i \in \omega$, $\Ccal$ has no non-empty $\Sigma^{1,X_i}_1$ subset,
$|a_{i+1}| \geq i$,
and for every $G \in \tuple{a_{i+1}, X_{i+1}}$, every~$H \in [G]^\omega$ and every~$j < i$, either~$\Ccal \cap [\Phi_j^H] = \emptyset$
or~$\Phi_j^H(\sigma) \downarrow = 1$ for some string 
$\sigma \in \omega^{<\omega}$ of length at least~$i$ such that $\Ccal \cap [\sigma] = \emptyset$.
Take $Y = \bigcup_i a_i$ as the desired set.
By construction, for every~$G \in [Y]^\omega$ and every~$j \in \omega$, either $\Ccal \cap [\Phi_j^G] = \emptyset$,
or $\{ \sigma \in \omega^{<\omega} : \Phi_j^G(\sigma) \downarrow = 1 \wedge \Ccal \cap [\sigma] = \emptyset \}$
is infinite. In either case, $[\Phi_j^G]$ is not a $G$-computably bounded subtree of~$\Ccal$.
\end{proof}

\begin{corollary}
$\wwkl \not \leq_{soc} \rt$.
\end{corollary}
\begin{proof}
Let~$T \subseteq 2^{<\omega}$ be a tree of positive measure
such that~$[T]$ has no non-empty $\Sigma^1_1$ subset.
Take for example $T$ to be a tree whose infinite paths are the elements of a $\Pi^{0,\Ocal}_1$ set of Martin-L\"of randoms relatively to Kleene's $O$. As the sets that are Turing below Kleene's $O$ are a basis for the $\Sigma^1_1$ subsets of $2^{\omega}$, $[T]$ cannot have any $\Sigma^1_1$ subset.


Fix an $\rt$-instance $f$ and suppose that every infinite $f$-homogeneous set~$H$ computes 
a infinite path through~$T$. In particular, $[T]$ has a non-empty $\Pi^{0,H}_1$ subset.
Since for every set~$X \in [\omega]^\omega$, there is an $f$-homogeneous set~$Y \in [X]^\omega$,
$[T]$ is $\Pi^0_1$ encodable.
Therefore, by Theorem~\ref{thm:comp-enc-compact-sigma11}, 
$[T]$ admits a non-empty $\Sigma^1_1$ subset, contradicting our hypothesis.
\end{proof}

Note that we make an essential use of compactness in Theorem~\ref{thm:comp-enc-compact-sigma11}.
Actually, there exist $\Pi^0_1$ encodable closed sets~$\Ccal \subseteq \omega^\omega$
with no $\Sigma^1_1$ subset, as witnesses the following lemma.

\begin{lemma}
Let~$Z \subseteq \omega$ be a set with no infinite set Turing below Kleene's $O$ subset in either it or its complement.
The set~$\Ccal_Z = \{ Y \in [\omega]^\omega : Y \subseteq Z \vee Y \subseteq \overline{Z} \}$
is $\Pi^0_1$ encodable and has no non-empty $\Sigma^1_1$ subset.
\end{lemma}
\begin{proof}
For any~$X \in [\omega]^\omega$, either $X \cap Z$, or~$X \cap \overline{Z}$ is infinite
and therefore belongs to~$\Ccal_Z$. Thus $\Ccal_Z$ is $\Pi^0_1$ encodable. Also as the sets Turing below Kleene's $O$ are a basis for the  $\Sigma^1_1$ subsets of $2^\omega$, $\Ccal_Z$ cannot have a non-empty $\Sigma^1_1$ subset.
%
%
\end{proof}

\section{Summary and open questions}

In this last section, we summarize the relations between variants of Ramsey's theorem
and of K\"onig's lemma over strong omniscient computable reducibility,
and state two remaining open questions.

\begin{center}
\begin{figure}[htbp]\label{fig:summary}
\begin{tikzpicture}[x=2cm, y=1.5cm, 
		node/.style={minimum size=2em, inner sep=3pt},
		arrow/.style={very thick,->}
]

	\node[draw, minimum width = 7cm, minimum height = 7cm] at (3, 3) {};
	\node[node] (kl) at (3, 5) {$\kl$};
	\node[node] (wkl) at (4, 5) {$\wkl$};
	\node[node] (wwkl) at (4, 4) {$\wwkl$};
	\node[node] (rt) at (3, 4) {$\rt$};
	\node[node] (rt23) at (3, 3) {$\rt^2_3$};
	\node[node] (rt22) at (4, 3) {$\rt^2_2$};
	\node[node] (srt2inf) at (2, 2) {$\srt^2_{<\infty}$};
	\node[node] (srt23) at (3, 2) {$\srt^2_3$};
	\node[node] (srt22) at (4, 2) {$\srt^2_2$};
	\node[node] (rt1inf) at (2, 1) {$\rt^1_{<\infty}$};
	\node[node] (rt13) at (3, 1) {$\rt^1_3$};
	\node[node] (rt12) at (4, 1) {$\rt^1_2$};

	\draw[arrow,<->] (kl) -- (wkl);
	\draw[arrow] (wkl) -- (wwkl);
	\draw[arrow] (kl) -- (rt);
	\draw[arrow] (rt22) -- (srt2inf);
	\draw[arrow] (rt23) -- (rt22);
	\draw[arrow, dotted] (rt) -- (rt23);
	\draw[arrow, dotted] (srt2inf) -- (srt23);
	\draw[arrow, dotted] (rt1inf) -- (rt13);
	\draw[arrow] (srt23) -- (srt22);
	\draw[arrow] (srt2inf) -- (rt1inf);
	\draw[arrow] (srt23) -- (rt13);
	\draw[arrow] (srt22) -- (rt12);
	\draw[arrow] (rt13) -- (rt12);

	\draw[arrow, thick] (rt22) -- (rt);

	\node[node, fill=white] (quest) at (3.5, 3.5) {?};
\end{tikzpicture}
\caption{Versions of $\rt$ and $\kl$ under~$\leq_{soc}$}
\end{figure}
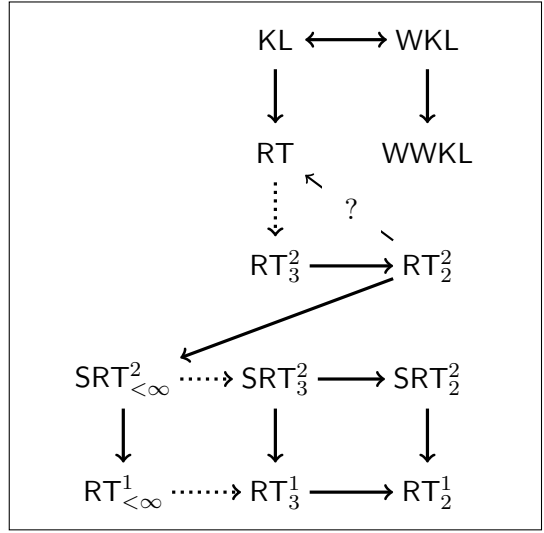
\end{center}

In Figure~\ref{fig:summary}, and plain arrow from~$\Psf$ to~$\Qsf$
means that~$\Qsf \leq_{soc} \Psf$. A dotted arrow indicates a hierarchy
between the statements. Except the open arrow from $\rt^2_2$ to~$\rt$,
the missing arrows are all known separations and can be derived from Section~\ref{subsect:kl-and-rt}.
The remaining questions are of two kinds: whether the number of colors and the size of the tuples
has a structural impact reflected over strong omniscient computable reducibility.

\begin{question}\label{quest:rtnkp-rtnk-soc}
Is $\rt^n_{k+1} \leq_{soc} \rt^n_k$ whenever~$n, k \geq 2$?
\end{question}

\begin{question}\label{quest:rtnpk-rtnk-soc}
Is $\rt^{n+1}_k \leq_{soc} \rt^n_k$ whenever~$n, k \geq 2$?
\end{question}

Note that a negative answer to Question~\ref{quest:rtnkp-rtnk-soc}
would give a negative answer to Question~\ref{quest:rtnpk-rtnk-soc}
since $\rt^n_{<\infty} \leq_{sW} \rt^{n+1}_2$ (see any of~\cite{Brattka2015Uniform,Hirschfeldtnotions}).

\vspace{0.5cm}

\noindent \textbf{Acknowledgements}.
The second author is funded by the John Templeton Foundation (`Structure and Randomness in the Theory of Computation' project). The opinions expressed in this publication are those of the author(s) and do not necessarily reflect the views of the John Templeton Foundation.

\end{document}